\documentclass[12pt, a4paper]{amsart}

\usepackage{amsfonts,amsmath,amssymb, amscd,fullpage}
\usepackage{stmaryrd}
\usepackage[all]{xy}

\newtheorem{theorem}{Theorem}[section]
\newtheorem{lemma}[theorem]{Lemma}
\newtheorem{proposition}[theorem]{Proposition}
\newtheorem{corollary}[theorem]{Corollary}
\newtheorem{question}[theorem]{Question}

\theoremstyle{definition}
\newtheorem{definition}[theorem]{Definition}

\theoremstyle{remark}
\newtheorem{remark}[theorem]{Remark}

\newtheorem{example}[theorem]{Example}

\newcommand\pf{\begin{proof}}
\newcommand\epf{\end{proof}}

\newcommand\C{\mathbb{C}}

\newcommand\yd{\mathcal{YD}}

\newcommand\ext{\mathrm{Ext}}

\newcommand\cd{\mathrm{cd}}
\newcommand\pd{\mathrm{pd}}
\newcommand\Tr{\mathrm{tr}}

\newcommand\ZZ{\mathbb{Z}}

\DeclareMathOperator{\Hom}{Hom}

\DeclareMathOperator{\GL}{GL}

\numberwithin{equation}{section}

\title{On the monoidal invariance of the cohomological dimension of  Hopf algebras}

\author{Julien Bichon}
\address{ Universit\'e Clermont Auvergne, CNRS, LMBP, F-63000 CLERMONT-FERRAND, FRANCE}
\email{julien.bichon@uca.fr}


\begin{document}

\begin{abstract}
	We discuss the question of whether the global dimension is a monoidal invariant for Hopf algebras, in the sense that if two Hopf algebras have equivalent monoidal categories of comodules, then their global dimensions should be equal. We provide several positive new answers to this  question, under various assumptions of   smoothness, cosemisimplicity or finite dimension. We also discuss the comparison between the global dimension and the Gerstenhaber-Schack cohomological dimension in the cosemisimple case, obtaining equality in the case the latter is finite. One of our main tools is the new concept of twisted separable functor.	 
	
\end{abstract}

\maketitle


\section{introduction}

A classical invariant of an algebra $A$ is its (right) global dimension
$${\rm r.gldim}(A) = {\rm max}\left\{\pd_A(M), \ M \in  \mathcal M_A\right\}   \in \mathbb N \cup \{\infty\}$$
where for a (right) $A$-module $M$, $\pd_A(M)$ stands for its projective dimension, i.e. the smallest possible length for a resolution of $M$ by projective $A$-modules. 

The global dimension is a key ingredient in the analysis of certain geometric properties of discrete groups \cite{br, dw}, and  often serves as a good analogue of the dimension of a smooth affine variety. However  in some noncommutative situations, it is better to replace it by the Hochschild cohomological dimension, which has similar geometric significance, and is defined by:
\begin{align*}
{\rm cd}(A)&= {\rm max}\left\{n : \ H^n(A, M) \not=0 \ {\rm for} \ {\rm some} \ A-{\rm bimodule} \ M\right\}\in \mathbb N \cup \{\infty\} \\
& = {\rm min}\left\{n : \ H^{n+1}(A, M) =0 \ {\rm for} \ {\rm any} \ A-{\rm bimodule} \ M\right\}  \\
&={\rm pd}_{_A\mathcal M_A}(A) 
\end{align*}
where $H^*(A,-)$ denotes Hochschild cohomology and ${\rm pd}_{_A\mathcal M_A}(A)$ is the projective dimension of $A$ in the category of $A$-bimodules. 

Indeed, for example if $A=A_1(k)$ is the first Weyl algebra ($k$ is, as in all the paper, an algebraically closed field), we have ${\rm r.gldim}(A_1(k)) = 1$ (in characteristic zero) and $\cd(A_1(k))=2$, while $A_1(k)$ should definitively be considered as a $2$-dimensional object.

When $A$ is a Hopf algebra, it is well-known that we have
 $${\rm r.gldim}(A) = \pd_A(k_\varepsilon)=\cd(A)= {\rm l.gldim}(A) = \pd_A(_\varepsilon k)$$
 where $k_\varepsilon$ and $_\varepsilon k$ denote the respective right and left trivial $A$-modules, and ${\rm l.gldim}(A)$ is the left global dimension. 
 See \cite{lolo} for the equalities at the extreme left and right, and, for example, \cite{giku}  for the other equality. We simply will denote this number by $\cd(A)$, and call it the cohomological dimension of $A$.
 
 A general classical problem is whether the global dimension or the Hochschild cohomological dimension remain preserved under various kind of ``deformations" of $A$, and the question we are particularly interested in, originally asked in \cite{bi16} and suggested by examples studied in \cite{bic}, is the following one.
 
 \begin{question}\label{ques}
 	If $A$ and $B$ are Hopf algebras having equivalent linear tensor categories of comodules, do we have $\cd(A)= \cd(B)$? 
 \end{question}

 Some remarks immediately arise on the signifance and interest of Question \ref{ques}.
\begin{enumerate}
	\item The word ``tensor" is crucial in the question, since this is what captures information about the algebra structure inside the category of comodules. Dropping it would make the question meaningless, as shown by the example of group algebras: if two group algebras have equivalent categories of comodules, the only conclusion, in lack of additional information, is that the groups have the same cardinality.
	\item Tannaka-Krein duality \cite{js} enables one to reconstruct a Hopf algebra from its tensor category of comodules together with the forgetful functor to vector spaces. However, it is not assumed here that the given monoidal equivalence is compatible with the respective forgetful functors, and so the Hopf algebras are non-isomorphic in general. There are many instances of the situation, see for example \cite{bic14,scsurv} for a large review of examples, and \cite{mro14,mro15,leta, rvdb} for more recent ones. 
	\item As just said, the Hopf algebras in Question \ref{ques} are non-isomorphic in general, but worst, some of their ring-theoretical properties, such as Gelfand-Kirillov dimension, can be very diffrerent,  see \cite{chiwawa}. The interest in the question is thus both theoretical, in the investigation of which properties of a Hopf algebra are preserved under monoidal equivalence of the category of comodules, and practical, in the determination of the global dimension of new Hopf algebras from known old ones.	
\end{enumerate}	

There are, to the best of our knowledge, two partial positive answers to Question \ref{ques} in the literature.

\begin{enumerate}
	\item In \cite{bi16,bi18}, it is shown that when $A$, $B$ are cosemisimple with antipode satisfying $S^4={\rm id}$, then $\cd(A)=\cd(B)$.
	\item In \cite{wyz},  Wang, Yu and Zhang show that when $A$ is twisted Calabi-Yau and $B$ is homologically smooth, then $\cd(A)=\cd(B)$.
\end{enumerate}


The aim of this paper is to provide several new positive answers to Question \ref{ques}, together with application to the determination of the cohomological dimension of some Hopf algebras in some new situations (universal cosovereign Hopf algebras and free wreath products). Indeed, we show that Question \ref{ques} has a positive answer in the following cases.

\begin{enumerate}
 \item The smooth case: we show that if $A$, $B$ have bijective antipode and are (homologically) smooth, then $\cd(A)=\cd(B)$. This improves on \cite[Theorem 2.4.5]{wyz}, which assumed moreover that $A$ is twisted Calabi-Yau (and then proved that $B$ is twisted Calabi-Yau as well), see Theorem \ref{thm:mismooth}. The proof is done by carefully inspecting the arguments in   \cite{wyz}. 
\item The cosemisimple case: we show that if $A$, $B$ are cosemisimple and both have finite cohomological dimension, then $\cd(A)=\cd(B)$. See Theorem \ref{thm:moninvcosemi}. Removing the assumption $S^4={\rm id}$ from \cite{bi18} (with instead the finiteness assumption on cohomological dimensions) enables us to compute cohomological dimension in a number of new situations,  see Section \ref{sec:exam}.
\item The finite-dimensional case: we show that under natural characteristic assumption on the base field or the assumption that $A^*$ is unimodular, then $\cd(A)=\cd(B)$, see Theorem \ref{thm:mifd}. Here, since finite-dimensional Hopf algebras are self-injective, we have $\cd(A) \in \{0, \infty\}$, and the interest of Question \ref{ques} is more on the theoretical side, but, as Etingof pointed out, understanding the finite-dimensional situation should be an important aspect. The proof of Theorem \ref{thm:mifd}  is a rather direct consequence of previous results \cite{lr,eg,aegn}, but an interesting aspect is that it connects Question \ref{ques} to a weak form of an important historical conjecture of Kaplansky saying that a finite-dimensional  cosemisimple Hopf algebra is unimodular (the strong form says that a cosemisimple Hopf algebra satisfies $S^2={\rm id}$).
\end{enumerate}




Our method in the smooth and cosemimple cases is based on the fact that if 	$\mathcal M^A \simeq^{\otimes} \mathcal M^B$ as above, results by Schauenburg \cite{sc1} ensures that there exists an $A$-$B$ Galois object $R$, and then on proving that $\cd(A)=\cd(R)=\cd(B)$, which is achieved in the smooth case by following arguments of Yu \cite{yu}. In general one notices furthermore that $\cd(A)=\pd_{ {_{R}^{}\mathcal M_R^{B}}}(R)$, the projective dimension of $R$ in the category of $R$-bimodules inside $B$-comodules, and then the main question is to compare $\pd_{{_{R}^{}\mathcal M_R^{B}}}(R)$ and $\pd_{ {_{R}^{}\mathcal M_R^{}}}(R)=\cd(R)$. The main ingredient in this comparison in the cosemisimple case is a twisted averaging trick, Lemma \ref{lem:averagebimod}, that we believe to be quite non-straightforward. The averaging lemma leads to the concept of twisted separable functor we define in Section \ref{sec:tsf}, a generalization of the notion of separable functor introduced in \cite{nvdbvo}.

Of course, the above considerations lead to the following question.

\begin{question}\label{ques:galois}
Let $A$ be a Hopf algebra. Under which conditions on $A$ do we have $\cd(A)=\cd(R)$ for any     left or right $A$-Galois object $R$?
\end{question}

Theorem \ref{thm:moninvcosemi} in the cosemimple case has the drawback, in concrete situations, that we need to know in advance that both Hopf algebras have finite cohomological dimension, an information that is not necessarily avalaible.  This leads us back to our initial idea to tackle Question \ref{ques} in \cite{bi16}, which was to use an auxiliary cohomological dimension for the Hopf algebra $A$, the Gestenhaber-Schack cohomological dimension,
defined by 
$${\rm cd}_{\rm GS}(A)= {\rm max}\{n : \ \ext^n_{\yd_A^{A}}(k,V)\not=0 \ {\rm for} \ {\rm some} \ V \in \yd_A^A\}\in \mathbb N \cup \{\infty\}$$  
where $\yd_A^{A}$ is the category of Yetter-Drinfeld modules over $A$ and $k$ is the trivial Yetter-Drinfeld module.
It was shown in \cite[Theorem 5.6, Corollary 5.7]{bi16} that ${\rm cd}(A) \leq {\rm cd}_{\rm GS}(A)$ and that
 if $A$, $B$ are Hopf algebras with $\mathcal M^A \simeq^\otimes \mathcal M^B$, then $$\max({\rm cd}(A),{\rm cd}(B))\leq {\rm cd}_{\rm GS}(A)={\rm cd}_{\rm GS}(B)$$ 
 Therefore, comparing $\cd(A)$ and $\cd_{\rm GS}(A)$ can be a key step towards answers to Question \ref{ques}. In this direction, we show (Theorem \ref{thm:cd=cdgs}) that if $A$ is a cosemisimple Hopf algebra with $\cd_{\rm GS}(A)$ is finite, then $\cd(A)=\cd_{\rm GS}(A)$. Again the method of proof is based on a twisted averaging trick and uses an appropriate twisted separable functor.  Theorem \ref{thm:cd=cdgs} has, as a corollary, a weak form of Theorem \ref{thm:moninvcosemi}, which is probably sufficient in dealing with  numerous examples, see Corollary  \ref{cor:invcdgs}.



We expect that the equality $\cd(A)=\cd_{\rm GS}(A)$ holds for any cosemisimple Hopf algebra, but as already pointed in  \cite{bi16}, it cannot  hold for any Hopf algebra over any field, as we see by taking a semisimple non cosemisimple Hopf algebra over a field of positive characteristic, so we asked there whether the equality was true in characteristic zero. Etingof pointed out that it does not hold in characteristic zero even for the very simple example $A=k[x]$ with $x$ primitive. Hence we have now the following question.

\begin{question}
	What are the Hopf algebras such that $\cd(A)=\cd_{\rm GS}(A)$?
	\end{question}

The paper is organized as follows. 
Section \ref{sec:hgmon} recalls the connection between Hopf-Galois objects and monoidal equivalences and proves our first result on the monoidal invariance of the cohomological dimension, in the smooth case. Section \ref{sec:bimodcat} provides the necessary material on categories of bimodules inside categories of comodules.
Section \ref{sec:tsf} introduces the notion of twisted separable functor.This is used in Section \ref{sec:coss}
to prove Theorem \ref{thm:moninvcosemi}, our second partial positive answer to Question \ref{ques}, in the cosemimple case.
Section \ref{sec:ydcd} discusses the comparison bewteen cohomological dimension and Gerstenhaber-Schack cohomological dimension, together with the necessary material  on Yetter-Drinfeld modules. Section \ref{sec:subhopf} studies the behaviour of Gerstenhaber-Schack cohomological dimension under Hopf subalgebras in the cosemisimple case.
 Section \ref{sec:exam} is devoted to applications to some examples. 
Section \ref{sec:smoofd} discusses the finite-dimensional situation in Question \ref{ques}.  The reader only interested in this case might go directly to this section. The concluding Section \ref{sec:summ} summarizes the known positive anwers to Question \ref{ques}.

\medskip

\noindent
\textbf{Notations and conventions.}
We work over an algebraically closed field $k$. 
We assume that the reader is familiar with the theory of Hopf algebras and their tensor categories of comodules, as e.g. in \cite{egno,ks,mon}, and with the basics of homological algebra \cite{br,wei}.
If $A$ is a Hopf algebra, as usual, $\Delta$, $\varepsilon$ and $S$ stand respectively for the comultiplication, counit and antipode of $A$. We use Sweedler's notations in the standard way. The category of right $A$-comodules is denoted $\mathcal M^A$, the category of right $A$-modules is denoted $\mathcal M_A$, etc...  
The trivial (right) $A$-module is denoted $k_\varepsilon$. The set of $A$-module morphisms (resp. $A$-comodule morphisms) between two $A$-modules (resp. two $A$-comodules) $V$ and $W$ is denoted ${\rm Hom}_A(V,W)$ (resp. ${\rm Hom}^A(V,W)$).

\medskip

\noindent
\textbf{Acknowledgements.} I would like to thank Pavel Etingof for interesting discussions and pertinent remarks.

\section{Hopf-Galois objects and monoidal equivalences}\label{sec:hgmon}

\subsection{Hopf-Galois objects} Let $A$ be a Hopf algebra. Recall that a \textsl{left $A$-Galois object} is a non-zero left $A$-comodule algebra $R$ such that the canonical map
\begin{align*}
R \otimes R &\longrightarrow A \otimes R \\
x \otimes y &\longmapsto x_{(-1)} \otimes x_{(0)} y
\end{align*}
is bijective. Similarly a \textsl{right $A$-Galois object}  is a non-zero right $A$-comodule algebra such that the obvious analogue of the previous canonical map is bijective. If $B$ is another Hopf algebra, an \textsl{$A$-$B$-bi-Galois object} is an $A$-$B$-bicomodule algebra which is simultaneously left $A$-Galois and right $B$-Galois. See \cite{sc1,scsurv}.

As said in the introduction, it is important, in view of Question \ref{ques}, to determine whether a Hopf algebra and its left or right Galois object  have the same cohomological dimension, which lead us to Question \ref{ques:galois}, and for which we list a number of basic remarks.


\begin{remark}\label{rem:leqcd}
	Let $A$ be a Hopf algebra and let $R$ be a left or right $A$-Galois object. Then we have $\cd(R)\leq \cd(A)$. This follows from Stefan's spectral sequence \cite[Theorem 3.3]{ste}, or can be checked directly at the level of complexes defining Hochschild cohomlogy \cite[Theorem 7.12]{bic14}. See \cite[Lemma 2.2]{yu} as well.
\end{remark}

\begin{remark}\label{rem:taft}
	There is indeed the need of assumptions on $A$ in Question \ref{ques:galois}, as the example of the Taft algebra $H_n$ shows, which admits the matrix algebra $M_n(k)$ as a Galois object \cite{mas}, and for which we have $\cd(M_n(k))=0<\cd(H_n)=\infty$.
\end{remark}

\begin{remark}
 In the setting of Question \ref{ques:galois}, as the Weyl algebra example shows, which is a Galois object over $k[x,y]$, the good dimension to consider is indeed the Hochschild cohomological dimension, and not the global dimension. 
\end{remark}

Recall that an algebra is said to be (homologically) \textsl{smooth} if the trivial bimodule has a finite resolution by finitely generated projective bimodules. For Hopf algebras, this is equivalent to say that the trivial left or right $A$-module  has a finite resolution by finitely generated projective modules.

A partial positive anwer to Question \ref{ques:galois} was obtained by Yu \cite{yu}. Indeed, if $A$ is a Hopf algebra with bijective antipode and  $R$ be a left of right $A$-Galois object, \cite[Theorem 2.4.5]{yu}  states that if $A$ is twisted Calabi-Yau of dimension $d$, then so is $R$, and hence in particular $d=\cd(A)=\cd(R)$. Our first obervation is that, inspecting carefully the arguments in \cite{yu}, what is needed to ensure the equality of the cohomological dimensions is smoothness only. 

\begin{theorem}\label{thm:cdeqgaloissmooth}
Let $A$ be a Hopf algebra with bijective antipode, and let $R$ be a left of right $A$-Galois object.
If $A$ is smooth, then we have $\cd(A)=\cd(R)$
\end{theorem}

\begin{proof}
Since $A$ is homologically smooth, we have that $\cd(A)$ is finite,  hence $\cd(A) =   {\rm max}\{n : \ext^n_{_A\mathcal{M}}({_\varepsilon k}, F)\not=0 \ \text{for some free module $F$}\}$,  and  moreover the functor $\ext_{_A\mathcal{M}}^*({_\varepsilon k}, -)$ commutes with direct colimits, see e.g. \cite[Chapter VIII]{br}). Hence
$$\cd(A) = \pd_{_A\mathcal{M}}({_\varepsilon k}) =  {\rm max}\{n : \ \ext^n_{_A\mathcal{M}}({_\varepsilon k},A)\not=0 \ \}$$  
The algebra $R$	is homologically smooth since $A$ is,  by \cite[Lemma 2.4]{yu}, hence we have similarly
$$\cd(R) = \pd_{_R\mathcal{M}_R}(R) =  {\rm max}\{n : \ \ext^n_{_R\mathcal{M}_R}(R,R\otimes R)\not=0 \ \}$$  
We have by \cite[Lemma 2.2, Lemma 2.1]{yu}
$$\ext^*_{_R\mathcal{M}_R}(R,R\otimes R) \simeq \ext^n_{_A\mathcal{M}}({_\varepsilon k}, {_AA \otimes R}))\simeq  \ext^*_{_A\mathcal{M}}({_\varepsilon k}, {_AA) \otimes R}$$
where the first isomorphism is obtained from  \cite[Lemma 2.2, Lemma 2.1]{yu}, with the left $A$-module structure on  ${_AA \otimes R}$ being simply by multiplying in $A$ on the left,  and the second one follows from the smoothness of $R$.
Hence we obtain $\cd(A)=\cd(R)$. 

	If we start with a right Hopf-Galois object $R$ over $A$, it is well-known that $R^{\rm op}$ is a left $A$-Galois object in a natural way (if the antipode of $A$ is bijective), so that we can use the result for left $A$-Galois objects to conclude that $\cd(A)=\cd(R)$ as well.
\end{proof}	

\subsection{Monoidal equivalences} Let $A$, $B$ be Hopf algebras. Schauenburg \cite[Corollary 5.7]{sc1} has shown the equivalence of the following assertions:
\begin{enumerate}
	\item There exists an equivalence of monoidal categories 	$\mathcal M^A \simeq^{\otimes} \mathcal M^B$;
	\item There exists an $A$-$B$-bi-Galois object.
\end{enumerate}

It therefore follows that finding answers to Question \ref{ques:galois} has immediate applications to Question \ref{ques}. We thus obtain, via Theorem \ref{thm:cdeqgaloissmooth}, a partial positive answer to Question \ref{ques}, only assuming that the Hopf algebras are smooth, while \cite[Theorem 2.4.5]{wyz} assumed furthermore that one of the Hopf algebras is twisted Calabi-Yau, and then proved that the other one is twisted Calabi-Yau with the same dimension as well.



\begin{theorem}\label{thm:mismooth}
	Let  $A$, $B$ be Hopf algebras  that have equivalent  linear tensor categories of comodules: 
	$\mathcal M^A \simeq^{\otimes} \mathcal M^B$. If $A$ and $B$ are smooth and have bijective antipode,  we have $\cd(A)=\cd(B)$.
\end{theorem}

\begin{proof}
Since there exists  an $A$-$B$-bi-Galois object $R$, we have $\cd(A)=\cd(R)=\cd(B)$ by Theoren \ref{thm:cdeqgaloissmooth}.
	\end{proof}

\begin{remark}
It is pointed out in Remark \ref{rem:taft} that the matrix algebra $M_n(k)$ is a Galois object over the Taft algebra $H_n$, and in fact $M_n(k)$ is an $H_n$-$H_n$-bi-Galois object \cite{sc00}. This indicates that the approach via Hopf-Galois objects cannot cover all the possible situations in Question \ref{ques}.
\end{remark}

\section{Equivariant bimodule categories and projective dimensions}\label{sec:bimodcat}

In this section we explain how one can use bimodule categories in order to obtain informations on Question \ref{ques:galois} and hence on Question \ref{ques}.

 Let $A$ be  a Hopf algebra, let $R$ be a right $A$-comodule algebra (recall that this means that $R$ is an algebra in the monoidal category $\mathcal M^{A}$) and let $_{R}^{}\mathcal M_R^{A}$ be the category of $A$-bimodules in the category $\mathcal M^{A}$: the objects are the $A$-comodules $V$ with an $R$-bimodule structure having the Hopf bimodule  compatibility conditions
$$ (x\cdot v)_{(0) }\otimes  (x\cdot v)_{(1)}= x_{(0)}\cdot v_{(0)} \otimes x_{(1)} v_{(1)}, \ 
(v\cdot x)_{(0) }\otimes  (v\cdot x)_{(1)}= v_{(0)}\cdot x_{(0)} \otimes v_{(1)}x_{(1)}$$
for any $x\in R$ and $v\in V$. The morphisms are the $A$-colinear and $R$-bilinear maps. The category $ {_{R}^{}\mathcal M_R^{A}}$ is obviously abelian, and the tensor product of bimodules induces a monoidal strucure on it.

The following basic property is certainly well-known, and a straightforward verification.

\begin{proposition} \label{prop:adjbimod} Let $A$ be a Hopf algebra and let $R$ be a right $A$-comodule algebra.
	\begin{enumerate}	
		\item 	The forgetful functor $\Omega^{A} : {_{R}^{}\mathcal M_R^{A}} \to \mathcal M^{A}$ has a left adjoint, which associates to a comodule $V$ the object $R\otimes V \otimes R$ whose bimodule structure is given by left and right multiplication of $R$ and whose comodule structure is the tensor product of the underlying comodules. 
		\item 	The forgetful functor $\Omega_{R} : {_{R}^{}\mathcal M_R^{A}} \to {_{R}^{}\mathcal M_R^{}}$ has a right adjoint, which associates to an $R$-bimodule $V$ the object $V\otimes A$ whose $R$-bimodule structure is given by
		$$x\cdot (v\otimes a)= x_{(0)}\cdot v\otimes x_{(1)}a, \quad (v\otimes a)\cdot x = v\cdot x_{(0)} \otimes ax_{(1)}$$ 
		and whose $A$-comodule structure is induced by the comultiplication of $A$. 
	\end{enumerate}
\end{proposition}

Objects in  $ {_{R}^{}\mathcal M_R^{A}}$ that are images of the above left adjoint functor are called \textsl{free}, they are indeed free as bimodules. Any object in $ {_{R}^{}\mathcal M_R^{A}}$ is a quotient of a free object. 

As usual, if $\mathcal C$ is an abelian category having enough projective objects, the notation $\pd_{\mathcal{C}}(V)$ refers to the projective dimension of the object $V$, i.e. the smallest length of a resolution of $V$ by projective objects in $\mathcal{C}$, with, as well
$$\pd_{\mathcal{C}}(V) =  {\rm max}\{n : \ \ext^n_{\mathcal{C}}(V,W)\not=0 \ \text{for some object $W$ in $\mathcal C$}\}$$

The following corollary is a  direct consequence of Proposition \ref{prop:adjbimod} and of the standard properties of pairs of adjoint functors.

\begin{corollary}Let $A$ be a Hopf algebra and let $R$ be a right $A$-comodule algebra.
\begin{enumerate}
	\item The category $ {_{R}^{}\mathcal M_R^{A}}$ has enough injective objects, since  ${_{R}^{}\mathcal M_R^{}}$ has, and we have, for any object $V$ in $ {_{R}^{}\mathcal M_R^{A}}$ and any $R$-bimodule $W$:
	$$\ext^*_{ {_{R}^{}\mathcal M_R^{}}}(\Omega_R(V), W) \simeq \ext^*_{ {_{R}^{}\mathcal M_R^{A}}}(V, W \otimes A)$$
	\item If $\mathcal M^{A}$ has enough projective objects (in which case one says that $A$ is co-Frobenius), so has $ {_{R}^{}\mathcal M_R^{A}}$. In that case, the previous isomorphism ensures that  for any object $V$ in $ {_{R}^{}\mathcal M_R^{A}}$, we have
	$$\pd_{ {_{R}^{}\mathcal M_R^{}}}(\Omega_R(V)) \leq \pd_{ {_{R}^{}\mathcal M_R^{A}}}(V)$$
	\item If $A$ is cosemisimple, then  ${_{R}^{}\mathcal M_R^{A}}$ has enough projective objects, and the projective objects are the direct summands of the free ones.
\end{enumerate}
\end{corollary}

The connection between our problem and bimodules is now given by the following result.

\begin{proposition}\label{prop:cdequivgaloi}
 Let $A$ be a Hopf algebra, let $R$ be a left $A$-Galois object, and let $B$ a Hopf algebra such that $R$ is $A$-$B$-bi-Galois. Then the category ${_{R}^{}\mathcal M_R^{B}}$ has enough projective objects, and we have
$$\cd(A) =  \pd_{ {_{R}^{}\mathcal M_R^{B}}}(R)\geq  \pd_{ {_{R}^{}\mathcal M_R^{}}}(R)=\cd(R)$$
\end{proposition}

\begin{proof}
 First recall that it follows from the right version of \cite[Theorem 5.7]{sc94} (the structure theorem for Hopf modules) that the functor 
\begin{align*}
_A\mathcal M &\longrightarrow {_A\mathcal M_A^{A}} \\
V &\longmapsto V \circledcirc A
\end{align*} 
is a monoidal equivalence of categories, where $V\circledcirc A$ is $V\otimes A$ as a vector space, has the tensor product left $A$-module structure and the right module and comodule structures are induced by the multiplication and comultiplication of $A$ respectively. This monoidal equivalence transforms the trivial module $_\varepsilon k$ into the $A$-bimodule $A$. 

Now let $R$ be an $A$-$B$-bi-Galois object.
The  corresponding monoidal equivalence $\mathcal{M}^{A} \simeq ^\otimes \mathcal M^B$ in \cite{sc1} is given by the cotensor product $-\square_A R$, and sends  $A$ to $R$, and thus clearly  induces an equivalence between the bimodule categories  ${_{A}^{}\mathcal M_A^{A}}$ and ${_{R}^{}\mathcal M_R^{B}}$. Composing with the equivalence at the beginning of the proof, we get a monoidal equivalence
	$$_A\mathcal M \simeq^{\otimes} {_{R}^{}\mathcal M_R^{B}}$$
	sending $_\varepsilon k$ to $R$. It is then clear that $  {_{R}^{}\mathcal M_R^{B}}$ has enough projective objects, and  that $\cd(A) = \pd_A({_\varepsilon k})= \pd_{ {_{R}^{}\mathcal M_R^{B}}}(R)$. The last inequality has been given in Remark \ref{rem:leqcd}. 
\end{proof}

It is therefore crucial to compare $\pd_{ {_{R}^{}\mathcal M_R^{B}}}(R)$ and $\pd_{ {_{R}^{}\mathcal M_R^{}}}(R)$ for $R$ a right $B$-Galois object. Notice that the problem makes sense and is interesting for any comodule algebra, as soon as ${_{R}^{}\mathcal M_R^{B}}$ has enough projective projects. This is the motivation for the tools we develop in the next section.

\section{Twisted separable functors}\label{sec:tsf}

In this section we introduce the notion of twisted separable functor. 

If $\mathcal C$ is category, we say that a subclass $\mathcal{F}$ of objects of $\mathcal{C}$ is generating if for every object $V$ of $\mathcal{C}$, there exists an object $P$ of $\mathcal{F}$ together with an epimorphism $P \to V$.

\begin{definition}\label{def:tsf}
Let $\mathcal C$ and $\mathcal D$ be some categories. We say that a functor $F : \mathcal{C} \to \mathcal{D}$ is \textsl{twisted separable} if there exist
\begin{enumerate}
	\item an autoequivalence $\Theta$ of the category $\mathcal{D}$;
	\item a generating subclass $\mathcal{F}$ of objects of $\mathcal{C}$  together with,  for any object $P$ of $\mathcal{F}$, an isomorphism $\theta_P : F(P)\to \Theta F(P)$;
	\item a natural morphism $\textbf{M}_{-,-} : \Hom_{\mathcal{D}}(F(-),\Theta F(-)) \to \Hom_{\mathcal{C}}(-,-)$ such that for any object $P$ of $\mathcal{F}$, we have $M_{P,P}(\theta_P)= {\rm id}_P$.
\end{enumerate}
\end{definition}

The naturality condition above means that for any morphisms $\alpha : V'\to V$, $\beta : W \to W'$ in $\mathcal{C}$ and any morphism $f : F(V) \to \Theta F(W)$ in $\mathcal{D}$, we have 
$$\beta \circ {\bf M}_{V,W}(f) \circ \alpha = {\bf M}_{V',W'}(\Theta F(\beta) \circ f \circ F(\alpha))$$

When $\mathcal{F}$ is the whole class of objects of $\mathcal C$, the autoequivalence $\Theta$ is the identity and the isomorphisms $\theta_P$ all are the identity, we get the notion of separable functor from \cite{nvdbvo}, which is known to be provide a convenient setting for various types of generalized Maschke theorems, see \cite{camizh}. A basic example of a separable functor is, when $A$ is a cosemisimple Hopf algebra, the forgetful functor $\mathcal{M}^{A} \to {\rm Vec}_k$: this is the content of Proposition \ref{prop:avecomod} in the next section.

Our motivation to introduce the present notion of twisted separable functor is the following result. 

\begin{proposition}\label{prop:pdtsf}
Let $\mathcal{C}$ and $\mathcal{D}$ be abelian categories having enough projective objects, and let $F : \mathcal{C} \to \mathcal{D}$ be a functor. Assume that the following conditions hold:
\begin{enumerate}
	\item the functor $F$ is exact and preserves projective objects;
	\item the functor $F$ is twisted separable and $\mathcal F$, the corresponding class of objects of  $\mathcal{C}$, contains a generating subclass $\mathcal F_0$ consisting of projective objects.
\end{enumerate}
Then,  for any object $V$ of $\mathcal{C}$ such that $\pd_{\mathcal{C}}(V)$ is finite, we have $\pd_{\mathcal{C}}(V) = \pd_{\mathcal{D}}(F(V))$. 
\end{proposition}

 We begin with some preliminaries. 

\begin{lemma}\label{lem:pdfree}
	Let $\mathcal{C}$ be an abelian category having enough projective objects, and let $\mathcal{F}_0$ be a generating subclass of $\mathcal{C}$ consisting of projective objects. If  $\pd_{\mathcal{C}}(V)$ is finite, we have 
	$$\pd_{\mathcal{C}}(V) =  {\rm max}\{n : \ \ext^n_{\mathcal{C}}(V,F)\not=0 \ \text{for some object $F$ in $\mathcal F_0$}\}$$  
\end{lemma} 

 \begin{proof}
 	Every object $X$ fits into an exact sequence
 	$0 \to W \to F \to X\to 0$
 	with $F$ an object of $\mathcal F_0$, hence projective. The result is thus obtained via a classical argument: if $n=\pd_{\mathcal{C}}(V)$, the long $\ext$ exact sequence gives that the functor $\ext_{\mathcal{C}}^n(V,-)$ is right exact, and hence $\ext_{\mathcal{C}}^n(V,F)\not = 0$  for some object $F$ of $\mathcal{F}_0$. 
 	\end{proof}

\begin{lemma}\label{lem:embedextF}
	Assume we are in the setting of Proposition \ref{prop:pdtsf}. For any objects $X,W$ of $\mathcal{C}$, we have a morphism
	$$\ext_{\mathcal{D}}^*(F(X),\Theta F(W)) \longrightarrow \ext_{\mathcal{C}}^*(X,W)$$
	which is surjective if $W$ is an object of $\mathcal F$.
\end{lemma}

\begin{proof}
	Start with  a projective resolution 
	$$\cdots  \longrightarrow P_n \overset{d_n}{\longrightarrow} P_{n-1} \overset{d_{n-1}}{\longrightarrow} \cdots \overset{d_2}{\longrightarrow}  P_1 \overset{d_1}{\longrightarrow} P_0 \overset{d_0}{\longrightarrow} X \to 0$$
	of $X$ by objects in $\mathcal{C}$. Since the functor $F$ is exact and preserves projectives, we get a projective resolution
	$$\cdots  \longrightarrow F(P_n) \overset{F(d_n)}{\longrightarrow} F(P_{n-1}) \overset{F(d_{n-1})}{\longrightarrow} \cdots \overset{F(d_2)}{\longrightarrow}  F(P_1) \overset{F(d_1)}{\longrightarrow} F(P_0) \overset{F(d_0)}{\longrightarrow} F(X) \to 0$$
	of $F(X)$ in $\mathcal D$. For all $i\geq 0$, we have, by the naturality assumption, commutative diagrams
	$$\xymatrixcolsep{5pc}
	\xymatrix{ 	
	\Hom_{\mathcal{D}}(F(P_i),\Theta F(W)) \ar[d]^-{{\bf M}_{P_i,W}} \ar[r]^-{- \circ F(d_{i+1})} & \Hom_{\mathcal D}(F(P_{i+1}), \Theta F(W)) \ar[d]^-{{\bf M}_{P_{i+1}, W}} \\
	\Hom_{\mathcal{C}}(P_i,W)  \ar[r]^-{-\circ d_{i+1}} & \Hom_{\mathcal C}(P_{i+1}, W) 
}$$
that induce a morphism of complexes
$$\widetilde{{\bf M}} : \Hom_{\mathcal D}(F(P_*), \Theta F(W)) \to \Hom_{\mathcal C}(P_*,W)$$ and hence a morphism between the corresponding cohomologies:
$$H^*(\widetilde{{\bf M}}): \ext_{\mathcal{D}}^*(F(X),\Theta F(W)) \longrightarrow \ext_{\mathcal{C}}^*(X,W)$$
Assume now that $W$ is an object of $\mathcal F$, and let $f \in \Hom_{\mathcal C}(P_i,W)$. We have 
$${\bf  M}_{P_i,W}(\theta_W \circ F(f))= M_{W,W}(\theta_W) \circ f= f$$
and if moreover $f\circ d_{i+1}=0$, we have also $\theta_W \circ F(f) \circ F(d_{i+1})=0$. This shows that $H^*(\widetilde{{\bf M}})$ is surjective.
\end{proof}

\begin{remark}
Assume, as the setting of Proposition \ref{prop:pdtsf} allows us to, that in the proof of the previous lemma, we have started with	 a projective resolution 
$$\cdots  \longrightarrow P_n \overset{d_n}{\longrightarrow} P_{n-1} \overset{d_{n-1}}{\longrightarrow} \cdots \overset{d_2}{\longrightarrow}  P_1 \overset{d_1}{\longrightarrow} P_0 \overset{d_0}{\longrightarrow} X \to 0$$
of $X$ by objects in $\mathcal{F}$. Then, for $f \in \Hom_{\mathcal C}(P_i,W)$, we have 
$${\bf  M}_{P_i,W}(\Theta(F(f))\circ \theta_{P_i})= f \circ M_{P_i,P_i}(\theta_{P_i}) = f$$
This shows that the morphism of complexes $\widetilde{{\bf M}} : \Hom_{\mathcal D}(F(P_*), \Theta F(W)) \to \Hom_{\mathcal C}(P_*,W)$ in the above proof is surjective in general. However, since we see no reason that $\Theta F(f) \circ \theta_{P_i} \circ F(d_{i+1})=0$, we cannot conclude that the corresponding morphism in cohomology is surjective without our assumption on $W$.
	\end{remark}

\begin{proof}[Proof of Proposition \ref{prop:pdtsf}]
Let $V$ be an object of $\mathcal C$, and let 
$$\cdots  \longrightarrow P_n \overset{d_n}{\longrightarrow} P_{n-1} \overset{d_{n-1}}{\longrightarrow} \cdots \overset{d_2}{\longrightarrow}  P_1 \overset{d_1}{\longrightarrow} P_0 \overset{d_0}{\longrightarrow} V \to 0$$
be a projective resolution of $V$. Since the functor $F$ is exact and preserves projectives, we get a projective resolution
$$\cdots  \longrightarrow F(P_n) \overset{F(d_n)}{\longrightarrow} F(P_{n-1}) \overset{F(d_{n-1})}{\longrightarrow} \cdots \overset{F(d_2)}{\longrightarrow}  F(P_1) \overset{F(d_1)}{\longrightarrow} F(P_0) \overset{F(d_0)}{\longrightarrow} F(V) \to 0$$
of $F(V)$ in $\mathcal D$. This shows that $\pd_{\mathcal{D}}(F(V)) \leq \pd_{\mathcal{C}}(V)$. To prove the converse inequality, we can assume that $n= \pd_{\mathcal{D}}(F(V))$ is finite.  We then have  in particular $\ext^{n+1}_{\mathcal{D}}(F(V), \Theta F(P))=\{0\}$ for any object $P$ in $\mathcal{F}$, and by Lemma \ref{lem:embedextF}, we have $\ext^{n+1}_{\mathcal{C}}(V, P)=\{0\}$ as well. Hence, assuming that $\pd_{\mathcal{C}}(V)$ is finite, Lemma \ref{lem:pdfree} shows that $\pd_{\mathcal{C}}(V)\leq n$, concluding the proof.
\end{proof}

In this paper we will not develop any more theory on twisted separable functors, and will focus on applications of Proposition \ref{prop:pdtsf}.

\section{Question \ref{ques} in the cosemisimple case}\label{sec:coss}

Recall that a Hopf algebra is cosemisimple if and only if it admits a Haar integral, i.e. a linear map $h: A \to k$ such that for any $a\in A$, we have
$$h(a_{(1)})a_{(2)}= h(a)= h(a_{(2)})a_{(1)} \quad \text{and} \quad h(1)=1$$
The proof of the semisimplicity of the category of comodules from the existence of a Haar integral is a consequence of the following classical averaging construction, which shows that the the forgetful functor $\mathcal{M}^{A} \to {\rm Vec}_k$ is separable, and that we record for future use.

\begin{proposition}\label{prop:avecomod}
Let $V$, $W$ be right $A$-comodules over a cosemisimple Hopf algebra $A$, and let $f: V\to W$ be a linear map. The map
\begin{align*}
{\bf M}_{V,W}(f) : V &\longrightarrow W \\
v & \longmapsto  h\left(f(v_{(0)})_{(1)}S(v_{(1)}) \right)f(v_{(0)})_{(0)} 
\end{align*}
is a morphism of comodules, with $M(f)=f$ if and only if $f$ is a morphism of comodules and with, for any morphisms of comodules $\alpha : V' \to V$ and $\beta : W \to W'$, 
$\beta \circ {\bf M}_{V,W}(f) \circ \alpha = {\bf M}_{V',W'}(\beta \circ f \circ \alpha)$. The above construction therefore defines a projection $${\bf M}_{V,W} : \Hom(V,W) \to \Hom^{A}(V,W)$$ that we call the averaging with respect to $V$ and $W$.
\end{proposition}

The Haar integral is not a trace in general, but satisfies a KMS type property, discovered by Woronowicz \cite{wo} in the setting of compact quantum groups.

\begin{theorem}\label{thm:modular}
Let $A$ be a cosemisimple Hopf algebra with Haar integral $h$. There exists a convolution invertible linear map $\psi : A \to k$, called a modular functional on $A$, satisfying the following conditions:
\begin{itemize}
	\item $S^2= \psi*{\rm id}*\psi^{-1}$; 
\item  $\sigma:=\psi*{\rm id}*\psi$ is an algebra automorphism of $A$;
\item   we have $h(ab)=h(b\sigma(a))$ for any $a, b \in A$.
\end{itemize}
\end{theorem}

The proof relies on the orthogonality relations, whose first occurence is due to Larson \cite{lar}, and were completed by Woronowicz \cite{wo}. In all the treatment we are aware of \cite{ks,pomu}, the setting is over the field of complex numbers, but inspecting the proof shows that it is valid for any cosemisimple Hopf algebra over any algebraically closed field.

We now present our key averaging lemma for bimodules. If $R$ is an $A$-comodule algebra over a cosemisimple Hopf algebra $A$, we denote by $\rho$ the algebra automorphism of $R$ defined by 
$\rho = {\rm id}*\psi^{-2}$, i.e. $\rho(x) = \psi^{-2}(x_{(1)})x_{(0)}$, with $\psi$ a modular functional as in Theorem \ref{thm:modular}.

\begin{lemma}\label{lem:averagebimod} Let $A$ be a cosemimple Hopf algebra and let $R$ be a right $A$-comodule algebra.
	Let $V,W$ be objects of  $_{R}^{}\mathcal M_R^{A}$. If $f : V \to W$ is a linear map satisfying 
	$$f(v\cdot x)= f(v)\cdot x \  \text{and} \ f(x\cdot v)= \rho(x)\cdot f(v)$$ for any $v\in V$ and $x \in R$, then ${\bf M}_{V,W}(f) : V \to W$ is a morphism in the category $_{R}^{}\mathcal M_R^{A}$.
\end{lemma}

\begin{proof}
	We already know that ${\bf M}_{V,W}(f) : V\to W$ is colinear and there remains to prove that ${\bf M}_{V,W}(f)$ is left and right $R$-linear as well. Let $v \in V$ and $x \in R$. We have, using our condition on $f$ and the compatibility between the comodule and right module structure:
	\begin{align*}
	{\bf M}_{V,W}(f)(v \cdot x) & = h\left(f((v\cdot x)_{(0)})_{(1)}S((v\cdot x)_{(1)})\right)f((v\cdot x)_{(0)})_{(0)}\\
	& = h\left(f(v_{(0)}\cdot x_{(0)})_{(1)}S(v_{(1)}x_{(1)})\right)f(v_{(0)} \cdot x_{(0)})_{(0)}\\
	& = h\left((f(v_{(0)})\cdot x_{(0)})_{(1)}S(v_{(1)}x_{(1)})\right)(f(v_{(0)}) \cdot x_{(0)})_{(0)}\\
	&=  h\left((f(v_{(0)})_{(1)} x_{(1)}S(v_{(1)}x_{(2)})\right)f(v_{(0)})_{(0)}\cdot x_{(0)}\\
	&=  h\left((f(v_{(0)})_{(1)} x_{(1)}S(x_{(2)})S(v_{(1)})\right)f(v_{(0)})_{(0)}\cdot x_{(0)}\\
	& ={\bf M}_{V,W}(f)(v)\cdot x
	\end{align*}
	Hence $f$ is right $R$-linear. We also  have, using our condition on $f$ and the compatibility between the comodule and left module structure:
	\begin{align*}
	{\bf M}_{V,W}(f)(x \cdot v) & = h\left(f((x\cdot v)_{(0)})_{(1)}S((x\cdot v)_{(1)})\right)f((x\cdot v)_{(0)})_{(0)}\\
	& =  h\left((f(x_{(0)} \cdot v_{(0)}))_{(1)}S(x_{(1)}v_{(1)})\right)f(x_{(0)}\cdot v_{(0)})_{(0)} \\
	& = \psi^{-2}(x_{(1)})h\left( (x_{(0)}\cdot f(v_{(0)}))_{(1)}S(x_{(2)}v_{(1)})\right) (x_{(0)}\cdot f(v_{(0)})_{(0)} \\
	& = \psi^{-2}(x_{(2)})h\left(x_{(1)}f(v_{(0)})_{(1)}S(x_{(3)}v_{(1)}))\right)x_{(0)}\cdot f(v_{(0)})_{(0)} \\
	\end{align*}
	Using the properties of the modular functional, this gives: 	
	\begin{align*}
	{\bf M}_{V,W}(f)(x \cdot v) &= \psi^{-2}(x_{(4)})h\left(f(v_{(0)})_{(1)}S(v_{(1)})S(x_{(5)})\psi(x_{(1)}) x_{(2)}\psi(x_{(3)})\right)x_{(0)} \cdot f(v_{(0)})_{(0)} \\
	&= h\left(f(v_{(0)})_{(1)}S(v_{(1)})S(x_{(4)})\psi(x_{(1)}) x_{(2)}\psi^{-1}(x_{(3)})\right)x_{(0)} \cdot f(v_{(0)})_{(0)} \\
	&= h\left(f(v_{(0)})_{(1)}v_{(1)}S(x_{(2)})S^2(x_{(1)})\right)x_{(0)} \cdot f(v_{(0)})_{(0)} \\
	&  = x \cdot {\bf M}_{V,W}(f)(v)
	\end{align*}
	and this shows that ${\bf M}_{V,W}(f)$ is left $R$-linear as well.			
\end{proof}


\begin{lemma}\label{lem:idrho} 
	Let $V$ be a comodule over a cosemisimple Hopf algebra $A$, let $R$ be an $A$-comodule algebra and consider the  map $\rho_V = \rho \otimes {\rm id}_{V }\otimes {\rm id}_R : R \otimes V\otimes R \to R \otimes V \otimes R$. We have ${\bf M}(\rho_V) = {\rm id}_{R \otimes V \otimes R}$, where ${\bf M}$ stands for averaging with respect to $ R \otimes V \otimes R$.	\end{lemma}

\begin{proof}
	It is immediate that 	 $\rho_V = \rho \otimes {\rm id}_{V }\otimes {\rm id}_R : R \otimes V\otimes R \to R \otimes V \otimes R$
	satisfies the assumption of Lemma \ref{lem:averagebimod}, hence ${\bf M}(\rho_V)$ is left and right $R$-linear. Since it is clear that ${\bf M}(\rho_V )(1\otimes v\otimes 1)= 1 \otimes v\otimes 1$ for any $v \in V$, we get the result by the $R$-bilinearity of ${\bf M}(\rho_V )$.
\end{proof}

We now have all the ingredients to prove the following result.

\begin{proposition}\label{prop:pdbimod}
	Let $A$ be  a cosemisimple Hopf algebra and let $R$ be a right $A$-comodule algebra. 
The forgetful functor $\Omega_{R} : {_{R}^{}\mathcal M_R^{A}} \to {_{R}^{}\mathcal M_R^{}}$ is twisted separable, and we have
	$\pd_{{_{R}^{}\mathcal M_R^{A}}}(V)=\pd_{{_{R}^{}\mathcal M_R^{}}}(V)$ for any object $V$ in ${_{R}^{}\mathcal M_R^{A}}$ such that $\pd_{ {_{R}^{}\mathcal M_R^{A}}}(V)$ is finite. In particular, if $\pd_{ {_{R}^{}\mathcal M_R^{A}}}(R)$ is finite, we have $\pd_{{_{R}^{}\mathcal M_R^{A}}}(R)=\pd_{{_{R}^{}\mathcal M_R^{}}}(R)=\cd(R)$.
\end{proposition}

\begin{proof}  In order to show that the forgetful functor  $\Omega_{R} : {_{R}^{}\mathcal M_R^{A}} \to {_{R}^{}\mathcal M_R^{}}$ is twisted separable, consider
	\begin{enumerate}
		\item the class $\mathcal{F}=\mathcal F_0$ of free bimodules in  ${_{R}^{}\mathcal M_R^{A}}$;
		\item the autoequivalence $\Theta$ of the category $ {_{R}^{}\mathcal M_R^{}}$ that associates to an $R$-bimodule $W$ the $R$-bimodule $_\rho W$ having $W$ as underlying vector space and $R$-bimodule structure given by $x\cdot'w\cdot'y= \rho(x)\cdot w\cdot y$, and is trivial on morphisms;
		\item for a free object $R\otimes V\otimes R$, the $R$-bimodule isomorphism $\rho_V : R \otimes V \otimes R \to {_\rho (R \otimes V \otimes R)}$ in Lemma \ref{lem:idrho};
		\item for objects $V, W$ in $ {_{R}^{}\mathcal M_R^{A}}$, the averaging map $${\bf M}_{V,W} : \Hom_{ {_{R}^{}\mathcal M_R^{}}}(V, {_\rho W}) \to \Hom_{ {_{R}^{}\mathcal M_R^{A}}}(V,W)$$ from Lemma \ref{lem:averagebimod}.
	\end{enumerate}
	It follows from Lemma \ref{lem:averagebimod}, Lemma \ref{lem:idrho} and Proposition \ref{prop:avecomod} that  the functor $\Omega_{R} : {_{R}^{}\mathcal M_R^{A}} \to {_{R}^{}\mathcal M_R^{}}$ 	is indeed  twisted separable. Moreover, as already said, the class $\mathcal F$ of free objects consists of projective objects, the projective objects in ${_{R}^{}\mathcal M_R^{A}}$ are direct summands of free objects and hence are preserved by $\Omega_R$, which is clearly exact. Hence we are in the situation of Proposition \ref{prop:pdtsf}, and we obtain the equality of projective dimensions.
\end{proof}



We obtain the following partial answer to Question \ref{ques:galois}.

\begin{theorem}\label{thm:cdhgcs}
Let $A$ be Hopf algebra and let $R$ be a left or right $A$-Galois object. If $A$ is cosemisimple and $\cd(A)$ is finite, we have $\cd(A)=\cd(R)$ 
\end{theorem}

\begin{proof} The result is obtained by combining Proposition \ref{prop:cdequivgaloi}  and Proposition \ref{prop:pdbimod}.
\end{proof}

We now obtain our partial answer to Question \ref{ques} in the cosemisimple case. The proof is similar to that of Theorem \ref{thm:mismooth}.

\begin{theorem}\label{thm:moninvcosemi}
	Let  $A$, $B$ be Hopf algebras  that have equivalent  linear tensor categories of comodules: 
	$\mathcal M^A \simeq^{\otimes} \mathcal M^B$. If $A$ and $B$ are cosemisimple and $\cd(A)$, $\cd(B)$ are finite,  we have $\cd(A)=\cd(B)$.
\end{theorem}


We finish the section by noticing that Proposition \ref{prop:pdbimod} can be strengthened in the case $S^4={\rm id}$.

\begin{proposition}\label{prop:bims4}		
Let $A$ be a cosemisimple Hopf algebra with $S^4={\rm id}$, and let $R$ be a right $A$-comodule algebra. 
\begin{enumerate}
\item The forgetful functor	 $\Omega_{R} : {_{R}^{}\mathcal M_R^{A}} \to {_{R}^{}\mathcal M_R^{}}$ is separable. We thus  have $\pd_{{_{R}^{}\mathcal M_R^{A}}}(V)=\pd_{{_{R}^{}\mathcal M_R^{}}}(V)$ for any object $V$ in ${_{R}^{}\mathcal M_R^{A}}$, and $\pd_{{_{R}^{}\mathcal M_R^{A}}}(R) =\cd(R)$. 

\item Let 	$F:\mathcal M^A \simeq^{\otimes} \mathcal M^B$ be a monoidal equivalence with $B$ satisfying $S^4={\rm id}$ as well. We then have, for the $B$-comodule algebra $T=F(R)$,  $\cd(R)=\cd(T)$.

\item Let $F:\mathcal M^{A} \longrightarrow {\rm Vec}_k$ be a fibre functor. If $\cd(R)$ is finite, we have, for the algebra $T=F(R)$, $\cd(R)=\cd(T)$.

\end{enumerate}
\end{proposition}

\begin{proof}
As in the proof of Lemma \ref{lem:averagebimod}, using the properties of the modular functional,	we see that for any $a,x \in A$
	$$h(S(a_{(1)})xa_{(2)}) = \psi^{-2}(a_{(2)})h\left(xa_{(3)}S^{-1}(a_{(1)})\right)
	$$ 
At $x=1$ this gives 	$\varepsilon(a)= \psi^{-2}(a_{(2)})h\left(a_{(3)}S^{-1}(a_{(1)})\right)$. If $S^4={\rm id}$, then $\psi^{-2}$ convolution commutes with the identity, hence we get $\psi^{-2} =\varepsilon$. Hence the automorphism $\rho$ associated to an $A$-comodule algebra $R$ is the identity, the autoequivalence $\Theta$ in the proof of Proposition \ref{prop:pdbimod} is the identity, and the class $\mathcal F$ is the class of all objects, and it follows that $\Omega_{R} : {_{R}^{}\mathcal M_R^{A}} \to {_{R}^{}\mathcal M_R^{}}$ is separable. The result about projective dimensions is then either well-known or follows from the obvious improvement of Proposition \ref{prop:pdbimod} in the separable case, having in mind that the conclusion of Lemma \ref{lem:embedextF} now holds for any object.	

A monoidal equivalence $F:\mathcal M^A \simeq^{\otimes} \mathcal M^B$ induces, as before, an equivalence  between the bimodule categories  ${_{R}^{}\mathcal M_R^{A}}$ and ${_{T}^{}\mathcal M_T^{B}}$ for $T=F(R)$, sending $R$ to $T$, and then the assumption $S^4={\rm id}$ on $A$ and $B$ ensures that $\cd(R) = \pd_{{_{R}^{}\mathcal M_R^{A}}}(R) = \pd_{{_{T}^{}\mathcal M_T^{B}}}(T) =\cd(T)$.

Start now with a fibre functor $F:\mathcal M^{A} \longrightarrow {\rm Vec}_k$, i.e. a $k$-linear monoidal exact faithful functor that commutes with colimits. Such a functor induces, by Tannaka-Krein duality (see e.g. \cite{js,egno}) or by the results in \cite{sc1}, a monoidal equivalence 
$\mathcal M^A \simeq^{\otimes} \mathcal M^B$ for some Hopf algebra $B$, with as well a monoidal equivalence ${_{R}^{}\mathcal M_R^{A}} \simeq^{\otimes} {_{T}^{}\mathcal M_T^{B}}$. The assumption that $S^4={\rm id}$ for $A$ then gives  $\pd_{{_{R}^{}\mathcal M_R^{A}}}(R) =\cd(R)$.  Since $\pd_{ {_{R}^{}\mathcal M_R^{A}}}(R) =\pd_{{_{T}^{}\mathcal M_T^{B}}}(T)$, Proposition \ref{prop:pdbimod} ensures, under the assumption that $\cd(R)$ is finite, that $\pd_{{_{T}^{}\mathcal M_T^{B}}}(T) = \cd(T)$, and thus this gives the expected result.
\end{proof}

\begin{example}
Let $\sigma : A\otimes A \to k$ be  (Hopf, right) $2$-cocycle on a Hopf algebra $A$ (see \cite{mon}), i.e. $\sigma$ is a convolution invertible linear map $\sigma : A\otimes A \to k$ satisfying, for any $a,b,c \in A$
$$\sigma(a,1)= \varepsilon(a)1= \sigma(1,a), \ \sigma(a_{(2)},b_{(2)}) \sigma(a_{(1)}b_{(1)},c) = \sigma(a,b_{(1)}c_{(1)}) \sigma(b_{(2)},c_{(2)})$$
If   $R$ is a right $A$-comodule algebra, we obtain a new (associative) algebra $R_\sigma$ by letting
$$x.y = \sigma(x_{(1)},y_{(1)}) x_{(0)}y_{(0)}$$
We then have, if $A$ is cosemisimple with $S^4={\rm id}$, $\cd(R)=\cd(R_\sigma)$ if $\cd(R)$ is finite.
\end{example}

\begin{proof}
	The algebra $R_\sigma$ is the image of $R$ under the fibre functor $\mathcal{M}^{A} \to {\rm Vec}_k$ which has the forgetful functor as underlying functor and monoidal constraint $V\otimes W \to V \otimes W$, $v\otimes w \mapsto \sigma(v_{(1)}, w_{(1)})  v_{(0)} \otimes w_{(0)}$. The result is thus a consequence of Proposition \ref{prop:bims4}.
	\end{proof}

\begin{remark}
	If the Hopf algebra $A^\sigma$ (see \cite{doi,sc1}) satisfies as well $S^4={\rm id}$, we can conclude that $\cd(R)=\cd(R_\sigma)$ without the finiteness assumption. This applies in particular, when $A$ is a group algebra, to  the $2$-cocycle twisting of a group-graded algebra, which therefore has the same Hochschild cohomological dimension as the original algebra. This was probably well-known, but we are not aware of an explicit reference for this fact.
\end{remark}

\section{yetter-drinfeld modules over cosemimple hopf algebras}\label{sec:ydcd}

Our aim in this section is to compare the cohomological dimension and the Gerstenhaber-Schack cohomological dimension of a cosemisimple Hopf algebra, providing in this way a version of Theorem \ref{thm:moninvcosemi} that looks slightly weaker, but that is probably more useful in concrete situations (Corollary \ref{cor:invcdgs}). 

 Recall that a (right-right) Yetter-Drinfeld
module over a Hopf algebra $A$ is a right $A$-comodule and right $A$-module $V$
satisfying the condition, $\forall v \in V$, $\forall a \in A$, 
$$(v \cdot a)_{(0)} \otimes  (v \cdot a)_{(1)} =
v_{(0)} \cdot a_{(2)} \otimes S(a_{(1)}) v_{(1)} a_{(3)}$$
The category of Yetter-Drinfeld modules over $A$ is denoted $\yd_A^A$:
the morphisms are the $A$-linear and $A$-colinear maps.  The category $\yd_A^{A}$ is obviously abelian,
and, endowed with the usual tensor product of 
modules and comodules, is a tensor category, with unit the trivial Yetter-Drinfeld module, denoted $k$.

The forgetful functor $\Omega^{A} :\yd_A^{A} \to \mathcal M^{A}$ has a left adjoint \cite{camizh97}, the free  Yetter-Drinfeld module functor, which sends a comodule $V$ to the Yetter-Drinfeld module $V\boxtimes A$, which as a vector space is $V \otimes A$, has the right module structure  given by multiplication on the right, and  right coaction given by
$$(v \otimes a)_{(0)}\otimes  (v \otimes a)_{(1)}=v_{(0)}  \otimes a_{(2)} \otimes S(a_{(1)})v_{(1)}a_{(3)}$$
A Yetter-Drinfeld module isomorphic to some $V\boxtimes A$ as above is said to be free. Let us record the following facts, that are straightforward consequences of standard properties of pairs of adjoint functors.
\begin{enumerate}
	\item Every Yetter-Drinfeld module is a quotient of a free Yetter-Drinfeld module. Indeed, for a Yetter-Drinfeld $V$, the $A$-module structure of $V$ induces a surjective morphism $\Omega^{A}(V)\boxtimes A \to V$. 
	\item If the category $\mathcal M^{A}$ has enough projective objects, then so has  $\yd_A^{A}$.
	\item If $A$ is cosemisimple, then $\yd_A^{A}$ has enough projective objects, and the projective objects are precisely the direct summands of the free Yetter-Drinfeld modules.
\end{enumerate}
Similarly, the forgetful functor $\Omega_{A} :\yd_A^{A} \to \mathcal M_{A}$ has a right adjoint \cite{camizh97}, the cofree  Yetter-Drinfeld module functor, which sends a module $V$ to the Yetter-Drinfeld module $V\# A$, which as a vector space is $V \otimes A$, has the right comodule structure  given by the comultiplication of $A$ on the right, and  right $A$-module structure given by
$$(v \otimes a)\cdot b=v\cdot b_{(2)} \otimes S(b_{(1)})ab_{(3)}$$
Again, as a consequence of general properties of adjoint functors, it follows that the category $\yd_A^{A}$ has enough injective objects,  since $\mathcal M_{A}$ has.

Recall that we have defined the Gerstenhaber-Schack cohomological dimension of a Hopf algebra $A$ by 
$${\rm cd}_{\rm GS}(A)= {\rm max}\{n : \ \ext^n_{\yd_A^{A}}(k,V)\not=0 \ {\rm for} \ {\rm some} \ V \in \yd_A^A\}\in \mathbb N \cup \{\infty\}$$  
The name comes from the fact, proved in \cite{tai04}, that if $V$ is a Yetter-Drinfeld over $A$, then
$\ext^*_{\yd_A^{A}}(k,V)$ is isomorphic with $H^*_{\rm GS}(A,V)$, the Gerstenhaber-Schack cohomology of $A$ with coefficients in $V$ \cite{gs1,shst}. 

Notice that since $\yd_A^{A}$ has enough injective objects, the above $\ext$ can computed using injective resolutions of $V$, and if $\yd_A^{A}$ has enough projective objects, using projective resolutions of $k$ in $\yd_A^{A}$. Another consequence of general properties of pairs of adjoint functors is that we have, for any Yetter-Drinfeld module $V$ and any $A$-module W, natural isomorphisms
$$\ext^*_A(\Omega_A(V),W) \simeq \ext^*_{\yd_A^{A}}(V, W\#A)$$
This is what proves that $\cd(A)\leq \cd_{\rm GS}(A)$ \cite{bi16}. 

We now present an  averaging lemma for Yetter-Drinfeld modules over cosemisimple Hopf algebras, in the same spirit as Lemma \ref{lem:averagebimod}, which will be the key tool towards the proof of Theorem \ref{thm:cd=cdgs}. If $A$ is a cosemisimple Hopf algebra with modular functional $\psi$, we denote by $\theta$ the algebra automorphism of $A$ defined by $\theta = \psi^2*{\rm id}$.

\begin{lemma}\label{lem:keyaverage}
	Let $V,W$ be Yetter-Drinfeld modules over a cosemisimple Hopf algebra $A$. If $f : V \to W$ is a linear map satisfying $f(v\cdot a)= f(v)\cdot \theta(a) $ for any $v\in V$ and $a \in A$, then ${\bf M}_{V,W}(f) : V \to W$ is a morphism of Yetter-Drinfeld modules.
\end{lemma}

\begin{proof} We already know that ${\bf M}_{V,W}(f) : V\to W$ is colinear and there remains to prove that ${\bf M}_{V,W}(f)$ is $A$-linear as well. Let $v \in V$ and $a \in A$. We have, using our condition on $f$ and the Yetter-Drinfeld property:
	\begin{align*}
	{\bf M}_{V,W}(f)(v \cdot a) & = h\left(f((v\cdot a)_{(0)})_{(1)}S((v\cdot a)_{(1)})\right)f((v\cdot a)_{(0)})_{(0)}\\
	& = h\left(f(v_{(0)}\cdot a_{(2)})_{(1)}S(S(a_{(1)})v_{(1)}a_{(3)})\right)f(v_{(0)} \cdot a_{(2)})_{(0)}\\
	& = h\left((f(v_{(0)})\cdot \theta(a_{(2)}))_{(1)}S(S(a_{(1)})v_{(1)}a_{(3)})\right)(f(v_{(0)}) \cdot \theta(a_{(2)}))_{(0)}\\
	&=  \psi^{2}(a_{(2)})h\left((f(v_{(0)})\cdot a_{(3)})_{(1)}S(S(a_{(1)})v_{(1)}a_{(4)})\right)(f(v_{(0)}) \cdot a_{(3)})_{(0)}\\
	& = \psi^{2}(a_{(2)})h\left(S(a_{(3)})f(v_{(0)})_{(1)}a_{(5)}S(S(a_{(1)})v_{(1)}a_{(6)})\right)f(v_{(0)})_{(0)} \cdot a_{(4)}\\
	& = \psi^{2}(a_{(2)})h\left(S(a_{(3)})f(v_{(0)})_{(1)}a_{(5)}S(a_{(6)})S(v_{(1)})S^2(a_{(1)})\right)f(v_{(0)})_{(0)} \cdot a_{(4)}\\
	& = \psi^{2}(a_{(2)})h\left(S(a_{(3)})f(v_{(0)})_{(1)}S(v_{(1)})S^2(a_{(1)})\right)f(v_{(0)})_{(0)} \cdot a_{(4)} 
	\end{align*}
Using the properties of the modular functional, and since $\sigma \circ S = \sigma^{-1} = \psi^{-1} *S * \psi^{-1}$ because $\sigma$ is an algebra map, this gives: 	
	\begin{align*}
	{\bf M}_{V,W}(f)(v \cdot a)	& = \psi^{2}(a_{(2)})h\left(f(v_{(0)})_{(1)}S(v_{(1)})S^2(a_{(1)})
	\sigma(S(a_{(3)})
	\right)f(v_{(0)})_{(0)} \cdot a_{(4)} \\
	& = \psi^{2}(a_{(2)})h\left(f(v_{(0)})_{(1)}S(v_{(1)})S^2(a_{(1)})
		\psi^{-1}(a_{(3)})S(a_{(4)})\psi^{-1}(a_{5})
		\right)f(v_{(0)})_{(0)} \cdot a_{(6)} \\
		& = h\left(f(v_{(0)})_{(1)}S(v_{(1)})S^2(a_{(1)})
		\psi(a_{(2)})S(a_{(3)})\psi^{-1}(a_{4}))
		\right)f(v_{(0)})_{(0)} \cdot a_{(5)} \\
		& =  h\left(f(v_{(0)})_{(1)}S(v_{(1)})S^2(a_{(1)}) S^3(a_{2})	
		\right)f(v_{(0)})_{(0)} \cdot a_{(3)} \\
		& =  h\left(f(v_{(0)})_{(1)}S(v_{(1)})
		\right)f(v_{(0)})_{(0)} \cdot a  = {\bf M}_{V,W}(f)(v)\cdot a
	\end{align*}
and this shows that ${\bf M}_{V,W}(f)$ is $A$-linear.	
\end{proof}



\begin{lemma}\label{lem:idtheta}
	Let $V$ be a right comodule over the cosemisimple Hopf algebra $A$, and consider the linear map $\theta_V = {\rm id}_V \otimes \theta : V\boxtimes A \to V \boxtimes A$. We have ${\bf M}(\theta_V) = {\rm id}_{V \boxtimes A}$, where ${\bf M}(\theta_V)$ stands for ${\bf M}_{V\boxtimes A,V\boxtimes A}(\theta_V)$.
\end{lemma}

\begin{proof}
It is immediate that 	${\rm id}_V \otimes \theta : V\boxtimes A \to V \boxtimes A$ satisfies the assumption of Lemma \ref{lem:keyaverage}, hence ${\bf M}({\rm id}_V \otimes \theta)$ is $A$-linear. Since it is clear that ${\bf M}({\rm id}_V \otimes \theta) (v\otimes 1)= v\otimes 1$ for any $v \in V$, we get the result by the $A$-linearity of ${\bf M}({\rm id}_V \otimes \theta)$.
	\end{proof}

We now have all the ingredients to prove 
the following result. 

\begin{proposition}\label{prop:pdyd}
Let $A$ be  a cosemisimple Hopf algebra. The forgetful functor $\Omega_{A} :\yd_A^{A} \to \mathcal M_{A}$ is twisted separable, and 
we have $\pd_{\yd_A^{A}}(V)=\pd_A(V)$ for any Yetter-Drinfeld module $V$ such that $\pd_{\yd_A^{A}}(V)$ is finite.
\end{proposition}

\begin{proof}  In order to show that the forgetful functor $\Omega_{A} :\yd_A^{A} \to \mathcal M_{A}$ is twisted separable, consider
	\begin{enumerate}
		\item the class $\mathcal{F}= \mathcal F_0$ of free Yetter-Drinfeld modules;
		\item the autoequivalence $\Theta$ of the category $\mathcal M_A$ that associates to a right $A$-module $W$ the $A$-module $W_\theta$ having $W$ as underlying vector space and $A$-module structure given by $w\cdot'a= w\cdot \theta(a)$, and is trivial on morphisms;
		\item for a free Yetter-Drinfeld module $V\boxtimes A$, the $A$-module isomorphism $\theta_V : V \otimes A \to (V \otimes A)_\theta$ in Lemma \ref{lem:idtheta}.
		\item for Yetter-Drinfeld modules $V, W$, the averaging map $${\bf M}_{V,W} : \Hom_A(V, W_\theta) \to \Hom_{\yd_A^{A}}(V,W)$$ from Lemma \ref{lem:keyaverage}.
	\end{enumerate}
It follows from Lemma \ref{lem:keyaverage}, Lemma \ref{lem:idtheta} and Proposition \ref{prop:avecomod} that  the functor	$\Omega_{A} :\yd_A^{A} \to \mathcal M_{A}$ is indeed  twisted separable. Moreover, as already said, the class $\mathcal F$ of free Yetter-Drinfeld modules consists of projective objects, the projective objects in $\yd_A^{A}$ are direct summands of free objects and hence are preserved by $\Omega_A$, which is exact. Hence we are in the situation of Proposition \ref{prop:pdtsf}, and we obtain the equality of projective dimensions.
	\end{proof}
	
We thus obtain the main result in the section.

\begin{theorem}\label{thm:cd=cdgs}
Let $A$ be a cosemisimple Hopf algebra. If $\cd_{\rm GS}(A)$ is finite, 	we have $\cd(A)=\cd_{\rm GS}(A)$.
\end{theorem}

\begin{proof} 
Let $A$ be a cosemimple Hopf algebra. Since $\cd_{\rm GS}(A) = \pd_{\yd_A^{A}}(k)$ and $\cd(A) = \pd_A(k_\varepsilon)$, we have $\cd(A)=\cd_{\rm GS}(A)$ if $\cd_{\rm GS}(A)$ is finite, by Proposition \ref{prop:pdyd}.
\end{proof}

We get the following weak form of Theorem \ref{thm:moninvcosemi}, whose formulation is useful.

\begin{corollary}\label{cor:invcdgs}
		Let  $A$, $B$ be Hopf algebras such that
	$\mathcal M^A \simeq^{\otimes} \mathcal M^B$. If $A$ and $B$ are cosemisimple and $\cd_{\rm GS}(A)$ is finite, we have $\cd(A)=\cd(B)$.
\end{corollary}

\begin{proof}
 We have $\cd_{\rm GS}(A) = \cd_{\rm GS}(B)$, hence $\cd(A)=\cd(B)$ by Theorem \ref{thm:cd=cdgs}.
\end{proof}

As in Section \ref{sec:bimodcat}, Proposition \ref{prop:pdyd} can be strengthened when $S^4={\rm id}$.

\begin{theorem}\label{thm:yds4}
	Let $A$ be Hopf algebra.
The forgetful functor $\Omega_{A} :\yd_A^{A} \to \mathcal M_{A}$ is  separable if and only if $A$ is cosemisimple and $S^4={\rm id}$, and in that case we have $\pd_{\yd_A^{A}}(V)=\pd_A(V)$ for any Yetter-Drinfeld module $V$. 
\end{theorem}

\begin{proof}
If $A$ is cosemisimple and $S^4={\rm id}$,	we see, as in the proof of Proposition \ref{prop:bims4}, that the automorphism $\theta$ of $A$ is the identity, and that $\Omega_{A} :\yd_A^{A} \to \mathcal M_{A}$ is  indeed separable, and the assertion on projective dimensions, which was already proved in \cite[Section 6]{bi18}, follows similarly.

Assume now that   $\Omega_{A} :\yd_A^{A} \to \mathcal M_{A}$ is  separable. Since $\Omega_{A}$ admits the right adjoint $-\#A$, the characterization of separability for functors that admit a right adjoint in \cite{raf} gives in particular an $A$-colinear and $A$-linear map
$$ \eta : k \# A \to k \ \text{with} \ \eta(1)=1$$
By the $A$-colinearity and $\eta(1)=1$, we have that $\eta=h$ is a Haar integral on $A$, which is thus cosemisimple.
The $A$-linearity of $h$ gives,
for any $a,x \in A$,
$$h(S(a_{(1)})xa_{(2)}) = \varepsilon(a)h(x)$$
We have seen in the proof of Proposition \ref{prop:bims4} that
for any $a,x \in A$,
$$h(S(a_{(1)})xa_{(2)}) = \psi^{-2}(a_{(2)})h\left(xa_{(3)}S^{-1}(a_{(1)})\right)
$$ 
Hence we have for any $a,x \in A$
$$h\left( x(\varepsilon(a) -  \psi^{-2}(a_{(2)})a_{(3)}S^{-1}(a_{(1)}))\right)= 0$$
The non-degeneracy of the Haar integral (which follows from the orthogonality relations) then gives, for any $a\in A$
$$\varepsilon(a)1 = \psi^{-2}(a_{(2)})a_{(3)}S^{-1}(a_{(1)})$$
Hence applying $\varepsilon$ gives
 $\varepsilon =  \psi^{-2}$, and we thus have $S^4={\rm id}$.
\end{proof}

We finish the section by noticing that Yetter-Drinfeld modules are also useful outside the cosemisimple case. Recall \cite{bi16} that a Yetter-Drinfeld module is said to be relative projective if it is a direct summand of a free Yetter-Drinfeld module, and let us say that a Hopf algebra is \textsl{Yetter-Drinfeld smooth} if the trivial object $k$ has a finite resolution by relative projective Yetter-Drinfeld modules that are finitely generated as modules. 

\begin{theorem}\label{thm:ydsmooth}
	Let  $A$, $B$ be Hopf algebras  that have equivalent  linear tensor categories of comodules: 
$\mathcal M^A \simeq^{\otimes} \mathcal M^B$. If $A$ and $B$ have bijective antipode and $A$ is Yetter-Drinfeld smooth, then we have $\cd(A)=\cd(B)$.
\end{theorem}

\begin{proof}
Clearly $A$ is smooth since it is Yetter-Drinfeld smooth, and if we start from a resolution of $k$ be finitely generated relative projective Yetter-Drinfeld modules in $\yd_A^A$, \cite[Theorem 4.1]{bic} ensures that one can transport this resolution to a resolution of $k$ to a finitely generated relative projective Yetter-Drinfeld modules in $\yd_B^B$. Hence $B$ is smooth as well and Theorem \ref{thm:mismooth} concludes the proof. 
	\end{proof}

\section{hopf subalgebras and cohomological dimension}\label{sec:subhopf}

Let $B\subset A$ be a Hopf subalgebra. Under the assumption of faithful flatness of $A$ as a $B$-module, which holds in many situations and in particular if $A$ is cosemisimple \cite{chi14}, we have $\cd(B) \leq  \cd(A)$ \cite[Proposition 3.1]{bi16}. In this section we prove, in view of an example in the next section,  an analogue inequality for Gerstenhaber-Schack cohomological dimension, in the cosemisimple case. Of course, if the  conclusion of Theorem \ref{thm:cd=cdgs} was known to hold for any cosemisimple Hopf algebra, this would become trivial.

We begin with some results of independent interest. Recall \cite{bi16} that a Yetter-Drinfeld module is said to be relative projective if it is a direct summand in a free one.

\begin{proposition}\label{prop:trivyd}
Let $A$ be a Hopf algebra, let $V$ be a Yetter-Drinfeld over $A$ and let $W$ be a right $A$-comodule. Then we have an isomorphism of Yetter-Drinfeld modules
$$(\Omega^A(V) \otimes W)\boxtimes A \simeq V\otimes (W\boxtimes A)$$
 In particular, if $P$ is a relative projective Yetter-Drinfeld module, so is the Yetter-Drinfeld module $V\otimes P$.
\end{proposition}

\begin{proof}
The map 
\begin{align*}
(\Omega^A(V)\otimes W)\boxtimes A &\longmapsto  V\otimes (W\boxtimes A) \\
v\otimes w \otimes a &\longmapsto v\cdot a_{(1)}\otimes w \otimes a_{(2)}
\end{align*}	
is easily seen to be a morphism of Yetter-Drinfeld modules, and its inverse is given by $v \otimes w\otimes a \mapsto  v\cdot S(a_{(1)})\otimes w \otimes a_{(2)}$. If $P$ is relative projective, let $W$ be a right $A$-comodule and $Q$ be a Yetter-Drinfeld module such that $W\boxtimes A \simeq P\oplus Q$. We then have $(V\otimes P) \oplus (V\otimes Q) \simeq V\otimes (W\boxtimes A) \simeq (\Omega^A(V)\otimes W)\boxtimes A$, which proves that $V\otimes P$ is relative projective.
	\end{proof}

\begin{corollary}\label{cor:cdgsglobal}
If $A$ is a cosemisimple Hopf algebra, we have
\begin{align*}
&{\rm cd}_{\rm GS}(A) = \pd_{\yd_A^{A}}(k)= {\rm max}\left\{n :  \ext^n_{\yd_A^{A}}(k, V) \not=0 \ {\rm for} \ {\rm some} \ V \in \yd_A^{A} \right\} 
\\
& = {\rm max}\left\{\pd_{\yd_A^{A}}(V), \ V \in  \yd_A^{A}\right\}
 =  {\rm max}\left\{n :  \ext^n_{\yd_A^{A}}(V, W) \not=0 \ {\rm for} \ {\rm some} \ V,W \in \yd_A^{A} \right\}\\
& = {\rm min}\left\{n : \ext^{n+1}_{\yd_A^{A}}(V, W) =0 \ {\rm for} \ {\rm any} \  V,W \in \yd_A^{A} \right\}  \\
&=  {\rm max}\left\{{\rm injd}_{\yd_A^{A}}(V), \ V \in  \yd_A^{A}\right\}
\end{align*}
where ${\rm injd}_{\yd_A^{A}}$ is the injective dimension in the category $\yd_A^{A}$.
\end{corollary}

\begin{proof}
The first two equalities have already been discussed. Let $P_* \to k$ be resolution of $k$ by projective objects, of length $n= \pd_{\yd_A^{A}}(k)$. Since $A$ is cosemisimple, the projective objects are the relative projectives, so if $V$ is a Yetter-Drinfeld module, tensoring 	the above resolution with $V$ yields, by Proposition \ref{prop:trivyd}, a length $n$ resolution of $V$ by projective objects. This gives the third equality, and the other ones then follow by classical arguments.
\end{proof}

Let $B\subset A$ be a Hopf subalgebra. Recall \cite{bi18} that there is a pair of adjoint functors
\begin{align*}
\yd^A_A & \longrightarrow \yd^B_B \quad \quad \yd^B_B  \longrightarrow \yd^A_A \\
X & \longmapsto X^{(B)} \quad \quad \quad V \longmapsto V\boxtimes_BA
\end{align*}
where
\begin{enumerate}
	\item for a Yetter-Drinfeld module $X$ over $A$, $X^{(B)}=\{x \in X \ | \ x_{(0)} \otimes x_{(1)} \in X \otimes B\}$ has the restricted $B$-module structure;
	\item for a Yetter-Drinfeld module $V$ over $B$,  $V\boxtimes_BA$ is the induced module $ V\otimes_BA$, with $A$-comodule structure given by
	$$(v \otimes_B a)_{(0)} \otimes (v\otimes_Ba)_{(1)} = v_{(0)} \otimes_B a_{(2)} \otimes S(a_{(1)}) v_{(1)} a_{(3)}$$
\end{enumerate}

\begin{lemma}\label{lem:restricteddirect}
Let $B\subset A$ be a Hopf subalgebra, and assume that $A$ is cosemisimple. Let $V$ be a Yetter-Drinfeld module over $B$. Then $V$ is isomorphic to a direct summand of $(V\boxtimes_BA)^{(B)}$.	
\end{lemma}

\begin{proof}
It is immediate to check that	we have a morphism of Yetter-Drinfeld modules $$i: V \to (V\boxtimes_BA)^{(B)}, \ v \mapsto v\otimes_B 1$$
Assume now that $A$ is cosemisimple. Then, by the proof of Theorem 2.1 in \cite{chi14}, there exists a sub-$B$-bimodule $T\subset A$, which is as well a subcoalgebra, such that $A = B \oplus T$.  Let $E : A \to B$ be the corresponding projection: $E(b)=b$ for $b \in B$ and $E(a)=0$ for $a \in T$. By construction $E$ is a $B$-bimodule map and a coalgebra map, and it is immediate to check that we have for any $a\in A$
$$S(E(a)_{(1)}) \otimes E(a)_{(2)} \otimes E(a)_{(3)} = S(a_{(1)}) \otimes E(a_{(2)}) \otimes a_{(3)} $$
From this, we see that the map 
$$ (V\boxtimes_BA)^{(B)}\to V, \ v\otimes_B a \to v\cdot E(a)$$
is a morphism of Yetter-Drinfeld modules. Since this map is clearly a retraction to $i$, this proves the lemma.
	\end{proof}

We now have all the ingredients to prove the expected result.

\begin{proposition}\label{prop:subcdgs}
Let $B\subset A$ be a Hopf subalgebra. If $A$ is cosemisimple, we have $\cd_{\rm GS}(B)\leq \cd_{\rm GS}(A)$.	
	\end{proposition}

\begin{proof}
We can assume that $\cd_{\rm GS}(A)=n$ is finite. Since $A$ is cosemisimple, \cite[Theorem 2.1]{chi14} ensures that $A$ is flat as a left $B$-module, and  $B\subset A$ is coflat. Hence,  by \cite[Proposition 3.3]{bi18} we have
$${\rm Ext}_{\yd_A^A}^*(V \boxtimes_B A, X) \simeq {\rm Ext}_{\yd_B^B}^*(V, X^{(B)})$$ 
for any Yetter-Drinfeld module $X$ over $A$, and any Yetter-Drinfeld module $V$ over $B$. Hence, for $V=k$, Corollary \ref{cor:cdgsglobal} yields
 $${\rm Ext}_{\yd_B^B}^{n+1}(k, X^{(B)})\simeq {\rm Ext}_{\yd_A^A}^{n+1}(k \boxtimes_B A, X)=\{0\} $$
 for  any Yetter-Drinfeld module $X$ over $A$.  Lemma \ref{lem:restricteddirect} ensures that any Yetter-Drinfeld module over $B$ is a direct summand in one of type $X^{(B)}$, so we get $\cd_{\rm GS}(B)\leq n$, as required.
\end{proof}

\section{examples}\label{sec:exam}

We now  use the previous results to examine some examples that were not  covered by  the literature.

\subsection{Universal cosovereign Hopf algebras} In this subsection we complete some of the results of \cite{bi18} on the cohomological dimension of the universal cosovereign Hopf algebras.
Recall that for $n \geq 2$ and  $F \in {\rm GL}_n(k)$,  the algebra $H(F)$ is the algebra
generated by
$(u_{ij})_{1 \leq i,j \leq n}$ and
$(v_{ij})_{1 \leq i,j \leq n}$, with relations:
$$ {u} {v^t} = { v^t} u = I_n ; \quad {vF} {u^t} F^{-1} = 
{F} {u^t} F^{-1}v = I_n,$$
where $u= (u_{ij})$, $v = (v_{ij})$ and $I_n$ is
the identity $n \times n$ matrix. The algebra
$H(F)$ has a  Hopf algebra structure
defined by
\begin{gather*}
\Delta(u_{ij}) = \sum_k u_{ik} \otimes u_{kj}, \quad
\Delta(v_{ij}) = \sum_k v_{ik} \otimes v_{kj}, \\
\varepsilon (u_{ij}) = \varepsilon (v_{ij}) = \delta_{ij}, \quad 
S(u) = {v^t}, \quad S(v) = F { u^t} F^{-1}.
\end{gather*}
We refer the reader to \cite{bi07,bi18} for more information and background on the universal cosovereign Hopf algebras $H(F)$. 

Recall \cite{bi18} that we say that a matrix $F \in \GL_n(k)$ is 

$\bullet$ normalizable if ${\rm tr}(F) \not= 0$ and  $ {\rm tr} (F^{-1})\not=0$ or ${\rm tr}(F)=0={\rm tr} (F^{-1})$;

$\bullet$ generic if it is normalizable and the solutions of the equation
$q^2 -\sqrt{{\rm tr}(F){\rm tr}(F^{-1})}q +1 = 0$ are generic, i.e. are not roots of unity of order $\geq 3$ (this property does not depend on the choice of the above square root);

$\bullet$ an asymmetry if there exists $E \in {\rm GL}_n(k)$ such that $F=E^tE^{-1}$.

\begin{theorem}
	Let $F\in \GL_n(k)$, $n\geq 2$. If $F$ is an asymmetry or $F$ is generic, we have $\cd(H(F)) = 3$.
\end{theorem}

\begin{proof}
	We know from \cite[Theorem 2.1]{bi18}, that  $\cd(H(F)) = 3$	if $F$ is an asymmery and  that $\cd_{\rm GS}(H(F) = 3$ if $F$ is generic, in which case $H(F)$ is cosemisimple \cite{bi07}, so Theorem \ref{thm:cd=cdgs} gives the result in that case.
\end{proof}

As an illustration of Theorem \ref{thm:cdhgcs},
consider, for $E \in \GL_n(k)$ and $F \in \GL_m(k)$, $n,m \geq 2$, the algebra
	 $H(E,F)$ presented by generators
	$u_{ij}$, $ v_{ij}$, $1\leq i \leq m, 1\leq j \leq n$,
	and relations
	$$u v^t = I_m = v F u^t E^{-1} \quad ; \quad
	v^tu = I_n = F u^t E^{-1} v.$$

\begin{theorem}
If $E$, $F$ are generic, $\Tr(E)=\Tr(F)$ and $\Tr(E^{-1})=\Tr(F^{-1})$, then we have $\cd(H(E,F))=3$.
\end{theorem}

\begin{proof}	
The assumption $\Tr(E)=\Tr(F)$ and $\Tr(E^{-1})=\Tr(F^{-1})$ ensures that $H(E,F)$ is an $H(E)$-$H(F)$-bi-Galois object \cite{bi07}. Hence, since the genericity assumption ensures that $H(E)$  cosemisimple and we know from the previous result that $\cd(H(E))$ and $\cd(H(F))$ are finite, the result follows from Theorem \ref{thm:cdhgcs}. 
\end{proof}

\subsection{Free wreath products} In this subsection we assume that the base field is $k=\C$, since the monoidal equivalences on which we rely \cite{fipi,leta} were obtained in this framework. Before going to the general setting of  Theorem \ref{thm:cdfw}, we feel it is probably worth to present a particular example. So for $n,p \geq 1$, consider, following the notation of \cite{bv}, the algebra $A_h^p(n)$  presented by generators $u_{ij}$, $1\leq i,j \leq n$, and  relations
$$\sum_{j=1}^nu_{ij}^p=1=\sum_{j=1}^nu_{ji}^p, \quad u_{ij}u_{ik}=0= u_{ji}u_{ki}, \ \text{for} \ k \not=j, $$
At $p=1$, $A_h^1(n)=A_s(n)$, the coordinate algebra of Wang's quantum permutation group \cite{wan98}. In general $A_h^p(n)$ is a Hopf algebra  with \cite{bic04}
 $$ \Delta(u_{ij})  = \sum_k u_{ik}\otimes u_{kj}, \ \varepsilon(u_{ij}) =\delta_{ij},  \ S(u_{ij})=u_{ji}^{p-1}$$

The following result, for which the $p=1$ case was obtained in \cite{bi16} (see \cite{bfg} as well, where it is shown that $A_s(n)$ is Calabi-Yau of dimension $3$), will be a  particular instance of the forthcoming Theorem \ref{thm:cdfw}.

\begin{theorem}\label{thm:cdahpn}
	We have, for $p\geq 1$ and $n\geq 4$, $\cd(A_h^p(n))=3$.
	\end{theorem}

Let $A$ be a Hopf algebra, and consider
$A^{*n}$, the free product algebra of $n$ copies of $A$, which inherits a natural Hopf algebra structure such that
the canonical morphisms $\nu_i : A \longrightarrow A^{*n}$
are Hopf algebras morphisms. The free wreath product $A*_w A_s(n)$ \cite{bic04}
 is the quotient of the algebra $A^{*n} * A_s(n)$ by the two-sided ideal generated by the elements:
	$$\nu_k(a) u_{ki} - u_{ki} \nu_k(a) \ , \quad 1 \leq i,k \leq n \ , \quad
	a \in A.$$
The free wreath product $A*_wA_s(n)$ admits a 
Hopf algebra structure given by
$$\Delta(u_{ij}) = \sum_{k=1}^n u_{ik}\otimes u_{kj} , \quad 
\Delta(\nu_i(a)) =   \sum_{k=1}^n \nu_i(a_{(1)}) u_{ik} \otimes
\nu_k(a_{(2)}), $$  
$$\varepsilon(u_{ij}) = \delta_{ij}, \quad \varepsilon(\nu_i(a)) = \varepsilon(a), \
S(u_{ij}) = u_{ji},  \quad  
S(\nu_i(a)) = \sum_{k=1}^n \nu_k(S(a))u_{ki}.$$
When $A$ is a compact Hopf algebra (i.e. arises from a compact quantum, we do not need the precise definition here), the free wreath product is as well a compact Hopf algebra. 
In that case the monoidal categories of comodules have been described for $n\geq 4$ by Lemeux-Tarrago \cite{leta} in the case $S^2={\rm id}$ and Fima-Pittau \cite{fipi} in general. 

Taking $A$ to be the group algebra $\C[\ZZ/p\ZZ]$, we have $A_h^p(n) \simeq \C[\ZZ/p\ZZ] *_w A_s(n)$ by \cite[Example 2.5]{bic04}, hence Theorem \ref{thm:cdahpn} is a particular instance of  the following result.


\begin{theorem}\label{thm:cdfw}
 We have 
	$\cd(A*_w A_s(n))= {\rm max}(\cd(A), 3)$ for any compact Hopf algebra $A$ such that $\cd(A)=\cd_{\rm GS}(A)$ and any $n\geq 4$.
\end{theorem}

\begin{proof}
	First notice that there is a Hopf algebra map $\pi : A*_w A_s(n) \to A_s(n)$ such that $\pi(u_{ij})= u_{ij}$ and $\pi(a) =\varepsilon(a)$, hence $A_s(n)$ stands as Hopf subalgebra of $A*_w A_s(n)$. We thus have, by \cite[Proposition 3.1]{bi16},  $3= \cd(A_s(n))\leq \cd(A*_w A_s(n))$.
	Similarly the natural map $A^{*n} \to A*_w A_s(n)$ has a retraction, and hence $A^{*n}$ stands as left coideal $*$-subalgebra of $A*_w A_s(n)$. By the results in \cite{chi}, $A*_w A_s(n)$ is thus faithfully flat as $A^{*n}$-module, hence projective \cite{mawi}. We then have, using \cite[Corollary 5.3]{bi18}, $\cd(A)=\cd(A^{*n}) \leq \cd(A*_w A_s(n))$, since restricting a resolution by projective $A*_wA_s(n)$-modules to $A^{*n}$-modules remains a projective resolution. Hence we have
	$$ {\rm max}(\cd(A), 3) \leq \cd( A*_w A_s(n))$$
	The converse inequality obviously holds is $\cd(A)$ is infinite, hence we can assume that $\cd(A)$ is finite, and hence, in view of our assumption, that $\cd_{\rm GS}(A)$ is finite.
	
	The results in \cite{fipi,leta} ensure the existence, for $q$ satisfying $q+q^{-1}=\sqrt{n}$, of a monoidal equivalence between the category of comodules over $A*_w A_s(n)$ and the category of comodules over a certain Hopf subalgebra $H$ of the free product $A*\mathcal O(SU_q(2))$.  
	We have, combining Proposition \ref{prop:subcdgs} and \cite[Corollary 5.10]{bi18}
	$$\cd_{\rm GS}(H)\leq  \cd_{\rm GS}(A*\mathcal O(SU_q(2))) = {\rm max}(\cd_{\rm GS}(A), \cd_{\rm GS}(\mathcal O(SU_q(2))) $$
	Since $\cd_{\rm GS}(\mathcal O(SU_q(2))=3$ by \cite{bic,bi16}, we get $\cd_{\rm GS}(H)\leq  {\rm max}(\cd_{\rm GS}(A), 3)$, and since we assume that $\cd_{\rm GS}(A)$ is finite, we get that $\cd_{\rm GS}(H)$ is finite.  Hence by Corollary \ref{cor:invcdgs} and Theorem \ref{thm:cd=cdgs}, we get
	$$\cd(A*_wA_s(n)) = \cd(H) = \cd_{\rm GS}(H) \leq {\rm max}(\cd_{\rm GS}(A), 3)= {\rm max}(\cd_{}(A), 3)$$
	which concludes the proof. \end{proof}

	

\begin{remark}
	At $n=2$, using the simple description of the free wreath product as a crossed coproduct in \cite{bic04}, it is not difficult to show directly that $\cd(A*_w A_s(2))= {\rm max}(\cd(A), 1)$ if $A$ is non trivial.
\end{remark}

\begin{remark}
	Fima-Pittau \cite{fipi} define more generally a free wreath product $A*_w A_{\rm aut}(R,\varphi)$, for suitable pairs $(R,\varphi)$ consisting of a finite-dimensional $C^*$-algebra and a faithful state, and prove a similar monoidal equivalence result, so that Theorem \ref{thm:cdfw} should generalize to this setting.
\end{remark}

\section{question \ref{ques} in the finite-dimensional case}\label{sec:smoofd}

In this section we provide a partial answer to Question \ref{ques} in the finite-dimensional case. Recall that a Hopf algebra $A$ is said to be unimodular if  there is a non-zero two-sided integral in $A$, i.e. there exists a non-zero $t \in A$ such that $ta=at= \varepsilon(a)t$ for any $a$. If $A$ is cosemisimple and finite-dimensional, then $A^*$ is unimodular.

\begin{theorem}\label{thm:mifd}
	Let  $A$, $B$ be finite-dimensional Hopf algebras such that
	$\mathcal M^A \simeq^{\otimes} \mathcal M^B$. Then we have $\cd(A)=\cd(B)$ if one of the following condition holds.
	\begin{enumerate}
		\item The characteristic of $k$ is zero, or satisfies $p > d^{\frac{\varphi(d)}{2}}$, where $d =\dim(A)$.
		\item $A^*$ is unimodular. 
	\end{enumerate}
\end{theorem}

\begin{proof}
	First notice that since a finite-dimensional Hopf algebra is self-injective (projective modules are injective), we have  $\cd(A), \cd(B) \in \{0, \infty\}$ and hence there are only few cases to consider. 
	Moreover, for the Drinfeld double $D(A)$, we have $\cd(D(A))=0$ if and only if $D(A)$ is semisimple, if and only if $A$ is semisimple and cosemisimple \cite[Proposition 7]{rad}, and $\cd(D(A))=\infty$ otherwise. Moreover, we have $\cd(D(A))=\cd(D(B))$ since our monoidal equivalence $\mathcal M^A \simeq^{\otimes} \mathcal M^B$ induces a monoidal equivalence between the monoidal centers of these categories (notice that  $\cd(D(A))=\cd_{\rm GS}(A)$).
	
	If $k$ has characteristic zero  or satisfies $p > d^{\frac{\varphi(d)}{2}}$, then by \cite[Theorem 3.3]{lr} and \cite[Theorem 4.2]{eg} respectively, we have that $A$ is semisimple if and only if $A$ is semisimple and cosemisimple, if and only if $\cd_{}(D(A))=0$. Hence under one of these assumptions we have $\cd(A)=\cd(B)$ because $\cd(D(A))=\cd(D(B))$.
	

	Since $\mathcal M^A \simeq^{\otimes} \mathcal M^B$ and $A$, $B$ are finite-dimensional, We have, by \cite[Corollary 5.9]{sc1},  $B\simeq A^\sigma$ for some Hopf $2$-cocycle $\sigma$. At the dual level this means that $B^* \simeq (A^*)^J$ for some Drinfeld twist $J$. Hence if $\cd(A)=0$, i.e. $A$ is semisimple, we have that $A^*$ is cosemisimple, and assuming that $A^*$ is unimodular, we have that $B^*$ is cosemisimple as well by \cite[Corollary 3.6]{aegn}, and hence $B$ is semisimple, so that $\cd(B)=0$, as needed. The assumption that $A^*$ is unimodular is stable under Drinfeld twist since the multiplication does not change, thus $B^*$ is unimodular as well, and hence we also have $\cd(B)=0 \Rightarrow \cd(A)=0$, concluding the proof.
\end{proof}

As we see in the proof of the previous theorem, a complete answer to Question \ref{ques} in the finite-dimensional case reduces to the question whether the class of finite-dimensional cosemisimple Hopf algebras is stable under Drinfeld twists. Remark 3.9 in \cite{aegn} claimed that this is expected to be true, and would follow from a weak form of an important conjecture of Kaplansky saying that a finite-dimensional  cosemisimple Hopf algebra is unimodular (the strong form says that a cosemisimple Hopf algebra satisfies $S^2={\rm id}$), but we are not aware of a proof since then.

\section{summary of known answers to Question \ref{ques}}\label{sec:summ}

In this last section, for the convenience of the reader, we summarize what are, to the best of our knowledge, the known positive answers to Question \ref{ques}. Let $A$, $B$ be Hopf algebras having equivalent linear tensor categories of comodules. Then we have  $\cd(A)= \cd(B)$ in the following situations.

\begin{enumerate}
 \item $A$, $B$ have bijective antipode and are smooth.
\item $A$, $B$ are cosemisimple and their antipodes satisfy $S^4={\rm id}$.
\item $A$, $B$ are cosemisimple and $\cd(A)$, $\cd(B)$ are finite.
\item $A$, $B$ are finite-dimensional, and the characteristic of $k$ is zero, or satisfies $p > d^{\frac{\varphi(d)}{2}}$, where $d =\dim(A)$.
\item $A$, $B$ are finite-dimensional and $A^*$ is unimodular.
\end{enumerate}

\end{document}